\documentclass[12pt]{article}
\usepackage{a4}
\usepackage{amsthm,amssymb,amsmath}
\usepackage{enumerate}
\usepackage{hyperref}
\usepackage{refcount}
\usepackage{graphicx}
\usepackage{cite}
\usepackage{tikz}
\usepackage{epsfig}

\newcommand\PP{{\mathbb P}}
\newcommand\NN{{\mathbb N}}
\newcommand\JJ{{\cal J}}

\newcommand{\floor}[1]{\lfloor #1\rfloor}

\newtheorem{theorem}{Theorem}
\newtheorem{lemma}[theorem]{Lemma}

\newtheorem{proposition}[theorem]{Proposition}
\newtheorem{conjecture}{Conjecture}
\newtheorem{problem}{Problem}

\newcounter{claimprefix}
\newtheorem{claim}{Claim}[claimprefix]

\newcommand*{\claimproof}{Proof of the Claim}
\newenvironment{proofcl}[1][\claimproof]{\begin{proof}[#1]
}{\end{proof}}

\begin{document}
\title{Density maximizers of layered permutations\thanks{The first, second and fourth authors were supported by the MUNI Award in Science and Humanities of the Grant Agency of Masaryk University. The first author was also supported by the project GA20-09525S of the Czech Science Foundation. The third author was supported by the Leverhulme Trust Early Career Fellowship ECF-2018-534.}}
\author{Adam Kabela\thanks{Faculty of Applied Sciences, University of West Bohemia,
Pilsen, Czech Republic. E-mail: {\tt kabela@kma.zcu.cz}. Previous affiliation: Faculty of Informatics, Masaryk University, Brno, Czech Republic.}\and
\and
Daniel Kr\'al'\thanks{Faculty of Informatics, Masaryk University, Botanick\'a 68A, 602 00 Brno, Czech Republic. E-mail: {\tt dkral@fi.muni.cz}.}
\and
Jonathan A. Noel\thanks{Department of Mathematics and Statistics, University of Victoria, Victoria, BC, Canada, V8P 5C2. E-mail: {\tt noelj@uvic.ca}. Previous affiliation: Mathematics Institute and DIMAP, University of Warwick, CV4 7AL Coventry, UK.} \and
Th\'eo Pierron\thanks{Univ. Lyon, Université Lyon 1, LIRIS UMR CNRS 5205, F-69621, Lyon, France. E-mail: {\tt theo.pierron@univ-lyon1.fr}. Previous affiliation: Faculty of Informatics, Masaryk University, Botanick\'a 68A, 602 00 Brno, Czech Republic.}
}
\date{}
\maketitle

\begin{abstract}
A permutation is layered if it contains neither $231$ nor $312$ as a pattern.
It is known that, if $\sigma$ is a layered permutation, then the density of $\sigma$ in a permutation
of order $n$ is maximized by a layered permutation. 
Albert, Atkinson, Handley, Holton and Stromquist [Electron. J. Combin. 9 (2002), R\#5]
claimed that the density of a layered permutation with layers of sizes $(a,1,b)$ where $a,b\geq2$
is asymptotically maximized by layered permutations with a bounded number of layers, and
conjectured that the same holds if a layered permutation has 
no consecutive layers of size one and its first and last layers are of size at least two.

We show that, if $\sigma$ is a layered permutation whose first layer is sufficiently large and second layer is of size one,
then the number of layers tends to infinity in every sequence of layered permutations asymptotically maximizing the density of $\sigma$.
This disproves the conjecture and the claim of Albert et al.
We complement this result by giving sufficient conditions on a layered permutation
to have asymptotic or exact maximizers with a bounded number of layers.
\end{abstract}

\section{Introduction}
\label{sec:intro}

We study permutations maximizing the density of a given pattern.
A \emph{permutation}
is a bijective function $\pi$ from $[n]$ to $[n]$ (we use $[n]$ to denote the set of the first $n$ positive integers);
the \emph{order} of a permutation $\pi$, denoted by $|\pi|$, is the size $n$ of its domain.
If $\pi$ is a permutation of order $n$,
then the permutation \emph{induced} by $m$ points $i_1<\cdots<i_m$ of $[n]$
is the unique permutation $\sigma:[m]\to [m]$ such that $\sigma(j)<\sigma(j')$ if and only if $\pi(i_j)<\pi(i_{j'})$ for all $j,j'\in [m]$.
The \emph{density} of a permutation $\sigma$ in a permutation $\pi$ of order $n\geq |\sigma|$, denoted by $d(\sigma,\pi)$,
is the probability that $|\sigma|$ randomly chosen distinct points of $[n]$ induce $\sigma$;
in this context, the permutation $\sigma$ is often referred to as a \emph{pattern} and
one may speak about permutations containing or avoiding a specific pattern.

The problem of maximizing the density of a permutation $\sigma$
can be traced back to the work of Galvin, Kleitman, Stromquist and Wilf in the early 1990s (cf.~\cite{Wil02}).
The \emph{packing density} of a permutation $\sigma$
is the limit of the maximum density of $\sigma$ in a permutation $\pi$ of order $n$ for $n\to\infty$;
an averaging argument implies that the limit exists as the quantity is non-increasing with $n$
(see, e.g.,~\cite[Proposition~1.1]{AlbAHHS02}).
It is easy to see that the packing density of the permutation $12$, and more generally $12\dots k$ for every $k\ge 2$, is $1$ (we write $\sigma(1)\cdots\sigma(m)$ to represent the permutation $\sigma$ of order $m$).
For the permutation $132$,
Galvin, Kleitman and Stromquist (independently and unpublished) showed that the packing density is $2\sqrt{3}-3$;
a sketch of a proof can be found in~\cite[Example 2.4.2]{War05}.
See e.g.~\cite{Mye02,SliS17} for additional results.

In this paper, we are interested in the structure of permutations maximizing the density of a layered permutation:
a permutation $\sigma$ of order $m$ is \emph{layered}
if $[m]$ can be partitioned to intervals $I_1,\ldots,I_k$ such that
$\sigma$ restricted to each $I_j$, $j\in [k]$, is decreasing, and
$\sigma(x)<\sigma(x')$ for any $x\in I_j$ and $x'\in I_{j'}$ such that $1\le j<j'\le k$ (we assume that the intervals are indexed in the order that they follow in $[m]$).
The intervals $I_1,\ldots,I_k$ are referred to as \emph{layers} of $\sigma$ and
layers of size one as \emph{singletons}.
We also refer to $\sigma$ as the layered permutation with layers of sizes $(|I_1|,\ldots,|I_k|)$.
It is not hard to show that a permutation is layered if and only if it avoids the patterns $231$ and $312$;
we refer to~\cite{BasC16,HilSV04,Pre08,War04,War06} for additional results on layered permutations.
If $\sigma$ is a layered permutation, then for every $n\ge|\sigma|$,
there exists a layered permutation of order $n$
that maximizes the density of $\sigma$ among all permutations of order $n$~\cite{AlbAHHS02,Bar04,Str93}, and
if no layer of $\sigma$ is a singleton,
then every permutation of order $n$ that maximizes the density of $\sigma$ is layered~\cite{AlbAHHS02};
we also refer to~\cite{Pri97} for a comprehensive treatment of maximizing the density of layered permutations.

To state our results, we need to fix some notation.
Given a permutation $\sigma$, a permutation $\pi$ of order at least $|\sigma|$ is \emph{$\sigma$-optimal}
if $\pi$ is a permutation of order $|\pi|$ with the largest density of $\sigma$
among all permutations of order $|\pi|$.
In particular, the results of the previous paragraph say that
if $\sigma$ is a layered permutation, then there exists a layered $\sigma$-optimal permutation of every order $n\geq|\sigma|$, and
if each layer of $\sigma$ has size at least two, then every $\sigma$-optimal permutation is layered.
Following the terminology of Warren~\cite{War05},
we say that a sequence $(\pi_n)_{n\in\NN}$ of permutations is \emph{near-$\sigma$-optimal}
if the orders of $\pi_n$ tend to infinity and
the density of $\sigma$ in $\pi_n$ converges to the packing density of $\sigma$.
Observe that for every layered permutation $\sigma$,
there exists a near-$\sigma$-optimal sequence of layered permutations.

Our central focus is on the question of
whether the number of layers in a near-$\sigma$-optimal sequence of layered permutations
must grow to infinity for a given layered permutation $\sigma$.
For example, if $\sigma$ is a layered permutation in which the first or last layer is a
singleton, 
then the number of layers in any near-$\sigma$-optimal sequence of layered permutations 
grows to infinity~\cite[Proposition~2.10]{AlbAHHS02}.
On the other hand,
if $\sigma$ has layers of sizes $(a,b)$ for $a,b\ge 2$,
then there exists a near-$\sigma$-optimal sequence of layered permutations where the number of layers is bounded~\cite{Pri97};
more generally, Albert et al.~\cite{AlbAHHS02} showed the following.

\begin{theorem}[{Albert et al.~\cite[Theorem~2.7]{AlbAHHS02}}]
\label{thm:thm27}
If $\sigma$ is a layered permutation such that each layer of $\sigma$ has size at least two, 
then there exist a near-$\sigma$-optimal sequence $(\pi_n)_{n\in\NN}$ and an integer $K\in\NN$ such 
that each $\pi_n$ has at most $K$ layers. 
\end{theorem}

Albert et al.~\cite{AlbAHHS02} also stated a conjecture generalizing Theorem~\ref{thm:thm27}.
We remark that the conjecture is stated in~\cite{AlbAHHS02} in terms of layered permutations of ``bounded-type.'' In~\cite[p.~13]{AlbAHHS02}, it is said that ``the bounded case corresponds precisely to when $p_s=\delta(\sigma)$ for some index $s$'' (in the notation of~\cite{AlbAHHS02}). This is equivalent to the existence of a near-$\sigma$-optimal sequence with a bounded number of layers, and so their conjecture is equivalent to the following.

\begin{conjecture}[{Albert et al.~\cite[Conjecture 2.9]{AlbAHHS02}}]
\label{conj:conj29}
If $\sigma$ is a layered permutation with layers of sizes $(\ell_1,\ldots,\ell_k)$ such that
$\ell_1>1$, $\ell_k>1$ and no two consecutive layers of $\sigma$ have size one,
then there exist a near-$\sigma$-optimal sequence $(\pi_n)_{n\in\NN}$ of layered permutations and
an integer $K\in\NN$ such that each $\pi_n$ has at most $K$ layers.
\end{conjecture}

As evidence for Conjecture~\ref{conj:conj29},
Albert et al.~\cite{AlbAHHS02} stated
the following proposition without proof and claimed that it follows by an argument similar to 
the proof of~\cite[Theorem~2.7]{AlbAHHS02}.

\begin{proposition}[{Albert et al.~\cite[Proposition~2.8]{AlbAHHS02}}]
\label{prop:prop28}
If $\sigma$ is a layered permutation with layers of sizes $(a,1,b)$ where $a>1$ and $b>1$,
then there exists a near-$\sigma$-optimal sequence $(\pi_n)_{n\in\NN}$ of layered permutations and
$K\in\NN$ such that each $\pi_n$ has at most $K$ layers.
\end{proposition}

Our first result is a counterexample to Conjecture~\ref{conj:conj29},
which also provides a counterexample to Proposition~\ref{prop:prop28}.

\begin{theorem}
\label{thm:main}
For every $k\in\NN$ and $\ell_1,\ldots,\ell_k$,
there exists $n_0$ such that
if $n\geq n_0$ and $\sigma$ is the layered permutation with layers of sizes $(n,1,\ell_1,\dots,\ell_k)$,
then the number of layers in the permutations of every near-$\sigma$-optimal sequence of layered permutations tends to infinity.
\end{theorem}

In particular, we show (see the discussion at the end of Section~\ref{sec:counterexample}),
that the layered permutation with three layers of sizes $(13,1,2)$ satisfies the conclusion of Theorem~\ref{thm:main},
and hence it fails to satisfy the conclusion of Proposition~\ref{prop:prop28} despite satisfying its hypothesis.
We remark that Proposition~\ref{prop:prop28} was used in the paper of H\"{a}st\"{o}~\cite{Has02}
for layered permutations with layers of sizes $(k,1,k)$ for $k\geq2$.
Fortunately, while Proposition~\ref{prop:prop28} is false in general,
its conclusion holds for this family of permutations.
This is a special case of
our positive results (Theorems~\ref{thm:finite} and~\ref{thm:constantBound2}) stated below.

We state and prove some of our results using the language from the theory of permutation limits
which was introduced by Hoppen et al.~\cite{HopKMRS13,HopKMS11a} and
applied in various problems concerning permutations e.g.~\cite{BalHLPUV15,ChaKNPSV20,GleGKK15,KraP13,Kur20}.
At the center of this theory is the notion of a \emph{permuton},
which is an analytic representation of a large permutation.
The link between the finite and limit settings is presented in Section~\ref{sec:limits}.
In particular,
we show that the existence of a near-$\sigma$-optimal sequence of layered permutations with bounded number of layers
is equivalent to the existence of a $\sigma$-optimal layered permuton with a bounded number of layers.

To complement the negative result of Theorem~\ref{thm:main},
we offer two sufficient conditions on a layered permutation $\sigma$ for the
existence of a near-$\sigma$-optimal sequences whose elements are layered permutations
with a bounded number of layers.
In fact, in some cases, we can obtain a seemingly stronger
conclusion that the number of layers in any $\sigma$-optimal permutation is bounded
by a universal constant depending on $\sigma$ only.
In Section~\ref{sec:concl}, we ask whether these properties are, in fact, equivalent. 
We remark that the importance of the presence of layers of size one in relation
to the existence of large $\sigma$-optimal permutations with a bounded number of layers
is consistent with additional results on layered permutations, 
cf.~\cite{Pri97, AlbAHHS02, Has02, War04, War06}.

\begin{theorem}
\label{thm:finite}
Let $\sigma$ be a layered permutation with layers of sizes $(\ell_1,\dots,\ell_k)$ and $k\ge 2$.
If $\ell_1\ge 2$, $\ell_k\ge 2$ and $\ell_i+\ell_{i+1}\ge\max\{\ell_1,\ell_k\}+1$ for every $i\in [k-1]$,
then every $\sigma$-optimal layered permuton has finitely many layers.
\end{theorem}

\begin{theorem}
\label{thm:constantBound2}
Let $\sigma$ be a layered permutation with layers of sizes $(\ell_1,\dots,\ell_k)$ and $k\ge 2$.
If $\ell_1=\ell_k\ge 2$ and
every two consecutive layers of $\sigma$ contain a layer of size $\ell_1=\ell_k$,
then there exists an integer $K\in\NN$ such that
the number of layers of any $\sigma$-optimal layered permutation is at most $K$.
\end{theorem}

\section{Permutation limits}
\label{sec:limits}

In this section, we introduce notions from the theory of permutation limits used throughout the paper, and 
establish a link between the number of layers of a pattern maximizer in the limit and finite settings.
Recall that the density of a permutation $\sigma$ in a permutation $\pi$ of order $n\geq|\sigma|$
is the probability that $|\sigma|$ randomly chosen distinct elements of $\pi$ induce $\sigma$.
A sequence $(\pi_n)_{n\in\NN}$ of permutations is \emph{convergent}
if their orders $|\pi_n|$ tend to infinity and $d(\sigma,\pi_n)$ converges for every permutation $\sigma$.
By compactness, 
every sequence of permutations with orders tending to infinity has a convergent subsequence.

Limits of convergent sequences of permutations can be represented by an analytic object,
which we now introduce;
a more thorough introduction can be found in~\cite{HopKMRS13,HopKMS11a},
also see~\cite{KraP13} where the limit object was viewed as the measure for the first time.
A \emph{permuton} is a probability measure $\Pi$ on the $\sigma$-algebra of Borel subsets from $[0,1]^2$ that
has uniform marginals, i.e.,
\[\Pi([a,b]\times[0,1]) = \Pi([0,1]\times[a,b]) = b - a\]
for all $0\le a\le b\le 1$.
If $\Pi$ is a permuton, then a \emph{$\Pi$-random permutation} $\sigma$ of order $m$
is obtained by sampling $m$ points according to $\Pi$,
sorting them according to their $x$-coordinates,
say $(x_1,y_1),\ldots,(x_m,y_m)$ for $x_1<\cdots<x_m$ (note that the $x$-coordinates are pairwise distinct with probability one), and
defining $\sigma$ so that $\sigma(i)<\sigma(j)$ if and only if $y_i<y_j$ for $i,j\in [m]$.
The \emph{density} of a permutation $\sigma$ in a permuton $\Pi$, denoted by $d(\sigma,\Pi)$,
is the probability that a $\Pi$-random permutation of order $|\sigma|$ is $\sigma$.

We say that a permuton $\Pi$ is a \emph{limit} of a convergent sequence $(\pi_n)_{n\in\NN}$ of permutations
if 
\[d(\sigma,\Pi)=\lim_{n\to\infty}d(\sigma,\pi_n)\]
for every permutation $\sigma$.
Every convergent sequence of permutations has a limit and this limit is unique.
On the other hand, 
a sequence of $\Pi$-random permutations with increasing orders is convergent 
and has $\Pi$ as its limit with probability one.

Recall that the \emph{support} of a measure on a set $X$
is the set of all points $x\in X$ such that every open set containing $x$ has a positive measure.
A permuton $\Pi$ is \emph{layered}
if there exists a (not necessarily finite) collection $\JJ$ of internally disjoint closed subintervals of $[0,1]$ such that
$(x,y)$ lies in the support of $\Pi$ if $x$ belongs to an interval $[a,b]\in\JJ$ and $y=b-(x-a)$,
i.e., the support of $\Pi$ is a union of intervals with slope $-1$ centered along the main diagonal of the square $[0,1]^2$ such that
their projections on the horizontal axis are internally disjoint.
Observe that the collection $\JJ$ (if it exists) is uniquely determined by the permuton $\Pi$.
We will also refer to intervals in $\JJ$ as the \emph{layers} of the permuton $\Pi$,
those with positive length as \emph{non-trivial layers}, and
those with length zero as \emph{trivial layers}. 
The \emph{number of layers} of a layered permuton $\Pi$
is the cardinality of $\JJ$;
note that the number of layers of a layered permuton may be (countably or uncountably) infinite. 
For example, the permuton supported on the line $y=x$ is layered with uncountably many trivial layers.
On the other hand, the permuton supported on the line $y=1-x$ has only one layer. An example of a layered permuton with countably many layers
is the permuton with layers $[1-2^{-k+1},1-2^{-k}]$, $k\in\NN$, and $[1,1]$ (note that
the point $(1,1)$ indeed belongs to the support).
See Figure~\ref{fig:layered_permuton} for the visualization of the three examples of layered permutons that we have just given and
some additional examples of layered permutons.
Observe that if the number of layers is finite, then each layer is non-trivial.
Analogously to the finite case, a permuton $\Pi$ is layered if and only if $d(231,\Pi)=0$ and $d(312,\Pi)=0$;
in particular, this can be derived from Lemma~\ref{lm:convergence} proven below.

\begin{figure}[!ht]
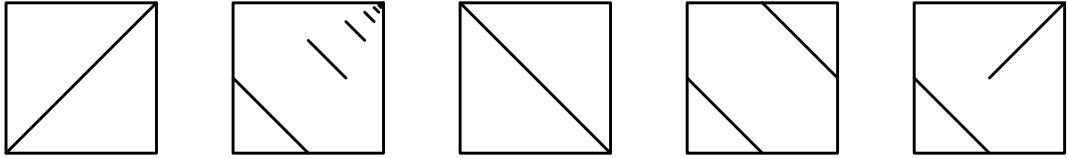

\begin{center}
\epsfbox{layer3-1.mps}
\hskip 8mm
\epsfbox{layer3-2.mps}
\hskip 8mm
\epsfbox{layer3-3.mps}
\hskip 8mm
\epsfbox{layer3-4.mps}
\hskip 8mm
\epsfbox{layer3-5.mps}
\end{center}
\caption{Examples of layered permutons.
         The support of each of the permutons is drawn in the square $[0,1]^2$.
	 The origin of the coordinate system is in the bottom left corner.}
\label{fig:layered_permuton}
\end{figure}

The lemma that we now state provides a key link between layered permutations and layered permutons.
We remark that it is not hard to construct a convergent sequence of layered permutations such that
the number of their layers tends to infinity
but the limit permuton has a finite number of layers;
for example, consider the sequence $\pi_1,\pi_2,\dots$ such that $\pi_n$ consists of one layer of length $n^2$ and  $n$ singleton layers. 

\begin{lemma}
\label{lm:convergence}
Let $(\pi_n)_{n\in\NN}$ be a convergent sequence of permutations and
let $\Pi$ be its limit permuton.
If the permutation $\pi_n$ is layered for every $n\in\mathbb{N}$, then $\Pi$ is layered.
Moreover, if there exists $k\in\NN$ such that the number of layers of each $\pi_n$ is at most $k$,
then the number of layers of $\Pi$ is at most $k$.
\end{lemma}

\begin{figure}[!ht]
\begin{center}
\epsfbox{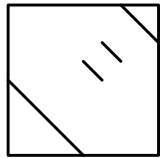}
\end{center}
\caption{The support of the layered permuton associated with the permutation $43215687$ as in the proof of Lemma~\ref{lm:convergence}.
	 The origin of the coordinate system is in the bottom left corner.}
\label{fig:convergence}
\end{figure}

\begin{proof}
We first associate each permutation $\pi_n$ with a permuton $\Pi_n$ as follows (see Figure~\ref{fig:convergence} for an example):
the permuton $\Pi_n$ is the unique (layered) permuton such that its support is the union 
\[\bigcup_{i\in[|\pi_n|]}\left\{\left(\frac{i-1+t}{|\pi_n|},\frac{\pi_n(i)-t}{|\pi_n|}\right),t\in [0,1]\right\}.\]
Informally speaking, the permuton $\Pi_n$ is the unique layered permuton
with the layers of the same relative size and order as in the permutation $\pi_n$.
Observe that the number of layers of each $\Pi_n$ is finite.

For $n\in\NN$ and $i\in\NN$,
let $J_{n,i}$ be the $i$-th longest layer of the permuton $\Pi_n$ if $i$ is at most the number of layers of $\Pi_n$ (order layers of the same length arbitrarily);
we set $J_{n,i}$ to be $[1,1]$ if $i$ exceeds the number of layers of the permuton $\Pi_n$.
Consider an increasing sequence $(n_k)_{k\in\NN}$ such that for every $i\in\NN$,
the sequences $(\min J_{n_k,i})_{k\in\NN}$ and $(\max J_{n_k,i})_{k\in\NN}$,
i.e., the sequences of the left and right end points of intervals $J_{n_k,i}$, converge;
the sequence $(n_k)_{k\in\NN}$ exists by compactness.
For $i\in\NN$, let $a_i$ and $b_i$ be the limits of the sequences $(\min J_{n_k,i})_{k\in\NN}$ and $(\max J_{n_k,i})_{k\in\NN}$, and
we define $\Pi$ to be the unique layered permuton whose support
is the closure of the set containing all points $(x,y)$ such that
\begin{itemize}
\item either $x$ is contained in a non-trivial interval $[a_i,b_i]$ for some $i\in\NN$ and $y=b_i-(x-a_i)$, or
\item $x$ is not contained in any interval $[a_i,b_i]$ and $y=x$.
\end{itemize}
We will show that the permuton $\Pi$ is the limit permuton of the sequence $(\pi_n)_{n\in\NN}$.

As the first step to show that $\Pi$ is the limit of $(\pi_n)_{n\in\NN}$,
we observe that 
\begin{equation}
\left|d(\sigma,\Pi_n)-d(\sigma,\pi_n)\right|\le\frac{|\sigma|^2}{|\pi_n|}\label{eq:condprob}
\end{equation}
for every permutation $\sigma$.
Indeed,
if $|\sigma|\le|\pi_n|$,
then the probability that $|\sigma|$ randomly chosen points of $\pi_n$ induce $\sigma$
is equal to the probability that the $\Pi_n$-random permutation of order $|\sigma|$ is $\sigma$
conditioned on no two of the $|\sigma|$ points defining the $\Pi_n$-random permutation
being chosen from the same interval $[(i-1)/|\pi_n|,i/|\pi_n|]$ for any $i\in [|\pi_n|]$.
Since the probability that two of the $|\sigma|$ points defining the $\Pi_n$-random permutation
are chosen from the same interval $[(i-1)/|\pi_n|,i/|\pi_n|]$ for some $i\in [|\pi_n|]$
is at most $\binom{|\sigma|}{2}|\pi_n|^{-1}$,
the estimate \eqref{eq:condprob} follows.

We next establish that
\begin{equation}
\lim_{k\to\infty}d(\sigma,\Pi_{n_k})=d(\sigma,\Pi)\label{eq:Pi}
\end{equation}
for every permutation $\sigma$.
By~\cite[Theorem~21]{Mel}, \eqref{eq:Pi} is equivalent to
\begin{equation}
\lim_{k\to\infty}\Pi_{n_k}(R)=\Pi(R)\label{eq:PiR}
\end{equation}
for every rectangle $R=[x_1,x_2]\times[y_1,y_2]\subseteq [0,1]^2$.
Observe that it is sufficient to establish~\eqref{eq:PiR}
for rectangles $R=[0,x]\times[0,y]\subseteq [0,1]^2$ only.
Fix $x$ and $y$ and assume that $x\le y$.
We distinguish two cases.
If $x$ is an internal point of an interval $[a_i,b_i]$ (see Figure~\ref{fig:PiR}),
then
\begin{equation}
\Pi(R)=\max\{a_i,x-\max\{b_i-y,0\}\}.\label{eq:fig_PiR}
\end{equation}

\begin{figure}[!ht]
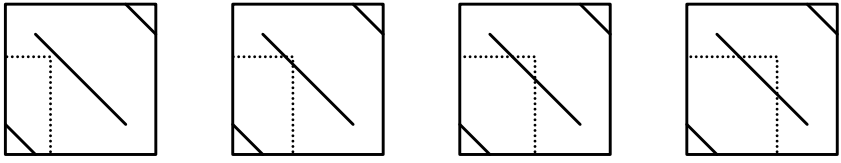

\begin{center}
\epsfbox{layer3-7.mps}
\hskip 8mm
\epsfbox{layer3-8.mps}
\hskip 8mm
\epsfbox{layer3-9.mps}
\hskip 8mm
\epsfbox{layer3-10.mps}
\end{center}
\caption{Illustration of the formula \eqref{eq:fig_PiR} for $[a_i,b_i]=[0.2,0.8]$. For $x\in\{0.3,0.4,0.5,0.6\}$ and $y=0.65$, the rectangle $R$ is drawn by the dotted lines.}
\label{fig:PiR}
\end{figure}

Note that $x$ is an internal point of the interval $J_{n_k,i}$ for all sufficiently large $n_k$ and
for all such $n_k$, it holds that
\[\Pi_{n_k}(R)=\max\{\min J_{n_k,i},x-\max\{\max J_{n_k,i}-y,0\}\}.\]
Hence, the identity \eqref{eq:PiR} follows.
On the other hand, if $x$ is not an internal point of an interval $[a_i,b_i]$ for any $i\in\NN$,
then $\Pi(R)$ and $\Pi_{n_k}(R)$ differ by at most the length of the interval $J_{n_k,i}$ that
contains $x$ as an internal point if such an interval exists, and
$\Pi(R)$ and $\Pi_{n_k}(R)$ are equal if there is no such interval.
Since the lengths of the intervals $J_{n_k,i}$ containing $x$ as an internal point must tend to zero (otherwise,
$x$ would be an internal point of an interval $[a_i,b_i]$ for some $i$),
the identity \eqref{eq:PiR} also follows in this case.
Since the case $x>y$ is analogous,
\eqref{eq:PiR} holds for every rectangle $R=[0,x]\times[0,y]\subseteq [0,1]^2$.
The identity \eqref{eq:Pi} now follows.

The estimate \eqref{eq:condprob} and the identity \eqref{eq:Pi} imply that
the permuton $\Pi$ is the limit permuton of the sequence $(\pi_{n_k})_{k\in\NN}$.
Since the sequence $(\pi_n)_{n\in\NN}$ is convergent, $\Pi$ is also its limit permuton.
The uniqueness of the limit permuton now yields the first part of the lemma.
For the second part, suppose that the number of layers of each permutation $\pi_n$ is at most $k$.
This implies that $J_{n_k,i}=[1,1]$ for all $i>k$ and so $\Pi$ has at most $k$ layers.
\end{proof}

A permuton $\Pi$ is \emph{$\sigma$-optimal}
if the density of $\sigma$ in $\Pi$ is equal to the packing density of $\sigma$.
Recall that, for every layered permutation $\sigma$ and every $n\geq|\sigma|$, there
exists a $\sigma$-optimal layered permutation $\pi_n$ of order $n$. By compactness,
the sequence $(\pi_n)_{n\in\NN}$ has a convergent subsequence.
By Lemma~\ref{lm:convergence},
the limit $\Pi$ of this subsequence is a layered permuton and, by the definition of limit permutons, 
the density of $\sigma$ in $\Pi$ is equal to the packing density.
Therefore, every layered permutation $\sigma$ has a $\sigma$-optimal layered permuton $\Pi$.
Moreover,
the density of $\sigma$ in $\Pi$ is the maximum possible density of $\sigma$ in any permuton:
if there was a permuton $\Pi'$ with a higher density of $\sigma$,
then a sequence of $\Pi'$-random permutations of increasing orders
would give a sequence of permutation with the density of $\sigma$ converging to a higher density than the packing density of $\sigma$.
Also note that
if $\Pi$ is a $\sigma$-optimal permuton with $k$ layers,
then a sequence of $\Pi$-random permutations of increasing orders is near-$\sigma$-optimal (with probability one) and
each of its elements has at most $k$ layers (since every $\Pi$-random permutation has at most $k$ layers).

\section{Maximizers with infinitely many layers}
\label{sec:counterexample}

This section is devoted to proving Theorem~\ref{thm:main}.
By Lemma~\ref{lm:convergence},
Theorem~\ref{thm:main} is implied by the next theorem,
which is the limit version of Theorem~\ref{thm:main}.

\begin{theorem}
\label{thm:counterexample}
For every $k\in\NN$ and $\ell_1,\ldots,\ell_k$,
there exists $n_0$ such that
if $\sigma$ is the layered permutation with layers of sizes $(n,1,\ell_1,\dots,\ell_k)$ where $n\ge n_0$,
then every $\sigma$-optimal layered permuton has an infinite number of layers.
\end{theorem}

\begin{proof}
\setcounter{claim}{0}
\setcounter{claimprefix}{\getrefnumber{thm:counterexample}}
Fix $k$ and $\ell_1,\ldots,\ell_k$ for the proof and set $L=1+\ell_1+\cdots+\ell_k$.
We choose the exact value of $n_0$ at the end of the proof;
until then, we will only need that $n_0\ge L$.
Fix $n\ge n_0$, i.e., the order of $\sigma$ is $n+L$, and
consider any $\sigma$-optimal layered permuton~$\Pi$.
Suppose, for the sake of contradiction, that $\Pi$ has finitely many layers.
Let $K$ be the number of layers of $\Pi$ and $x_i$ the length of the $i$-th layer, $i\in [K]$. Note that all 
layers of $\Pi$ are non-trivial and that $\Pi$ must have at least as many layers as $\sigma$; so, $K\geq k+2$. 
Observe that
\begin{equation}
\label{eq:Dequals}
d(\sigma,\Pi) = A\left(\sum_{1\leq a<b<i_1<\cdots <i_k\leq K} x_a^nx_b\prod_{j=1}^kx_{i_j}^{\ell_j}\right)
\end{equation}
where
\[A=\frac{(n+L)!}{n!\prod_{j=1}^k\left(\ell_j!\right)}.\]
As we will consider several modifications of the permuton $\Pi$,
the following shorthands for various quantities will be useful when calculating the density of $\sigma$ in the modified permutons:
\begin{align*}
\Sigma_0 & = \sum_{3\leq a<b<i_1<\cdots<i_k\leq K}x_a^nx_b\prod_{j=1}^kx_{i_j}^{\ell_j},\\
\Sigma_1 & = \sum_{3\leq b<i_1<\cdots<i_k\leq K}x_b\prod_{j=1}^kx_{i_j}^{\ell_j},\mbox{ and}\\
\Sigma_2 & = \sum_{3\leq i_1<\cdots<i_k\leq K}\prod_{j=1}^kx_{i_j}^{\ell_j}.
\end{align*}
Observe that $\Sigma_0 \leq \Sigma_1 \leq \Sigma_2\le 1$ and $\Sigma_2 >0$ because $K\geq k+2$.
Also note that
\begin{equation}
d(\sigma,\Pi) = A\left(x_1^nx_2\Sigma_2 + x_1^n\Sigma_1+x_2^n\Sigma_1 + \Sigma_0\right).
\label{eq:Dexpanded}
\end{equation}

We will next consider five permutons obtained by modifying $\Pi$,
relate the density of $\sigma$ in them to the density in $\Pi$ and
use this to derive some properties of $\Pi$;
the five permutons that we consider are visualized in Figure~\ref{fig:permutons}.

\begin{figure}
\begin{center}
  \begin{tikzpicture}[scale=2]
    \draw (0,0) -- (0,1) -- (1,1) -- (1,0) -- cycle;
    \draw[thick] (0,0.1) -- (0.1,0);
    \draw[thick] (0.1,0.4) -- (0.4,0.1);
    \draw[gray!50,dashed] (0.1,0) -- (0.1,0.4) -- (1,0.4);
    \draw[gray!50,dashed] (0,0.1) -- (0.4,0.1) -- (0.4,1);
    \node at (0.7,0.7) {$\Pi^-$};
    \node at (0.05,-0.1) {\tiny $x_2$};
    \node at (0.25,-0.1) {\tiny $x_1$};
    \node at (0.5,-0.3) {$\Pi'$};
    \tikzset{xshift=1.5cm}
    \draw (0,0) -- (0,1) -- (1,1) -- (1,0) -- cycle;
    \draw[thick] (0,0.2) -- (0.2,0);
    \draw[thick] (0.4,0.2) -- (0.2,0.4);
    \draw[gray!50,dashed] (0.2,0) -- (0.2,0.4) -- (1,0.4);
    \draw[gray!50,dashed] (0,0.2) -- (0.4,0.2) -- (0.4,1);
    \node at (0.7,0.7) {$\Pi^-$};
    \node at (0,-0.1) {\tiny $x_1+y$};
    \node at (0.4,-0.1) {\tiny $x_2-y$};
    \node at (0.5,-0.3) {$\Pi_y$};
    \tikzset{xshift=1.5cm}
    \draw (0,0) -- (0,1) -- (1,1) -- (1,0) -- cycle;
    \draw[thick] (0,0.4) -- (0.4,0);
    \draw[gray!50,dashed] (0,0.4) -- (1,0.4);
    \draw[gray!50,dashed] (0.4,0) -- (0.4,1);
    \node at (0.7,0.7) {$\Pi^-$};
    \node at (0.2,-0.1) {\tiny $x_1+x_2$};
    \node at (0.5,-0.3) {$\Pi_{x_2}$};
    \tikzset{xshift=1.5cm}
    \draw (0,0) -- (0,1) -- (1,1) -- (1,0) -- cycle;
    \draw[thick] (0,0.6) -- (0.6,0);
    \draw[gray!50,dashed] (0,0.6) -- (1,0.6);
    \draw[gray!50,dashed] (0.6,0) -- (0.6,1);
    \node at (0.8,0.8) {$\Lambda$};
    \node at (0.3,-0.1) {\tiny $\frac n{n+L}$};
    \node at (0.5,-0.3) {$\Pi^+$};
    \tikzset{xshift=1.5cm}
    \draw (0,0) -- (0,1) -- (1,1) -- (1,0) -- cycle;
    \draw[thick] (0,0.25) -- (0.25,0);
    \draw[gray!50,dashed] (0.25,0) -- (0.25,0.4) -- (1,0.4);
    \draw[gray!50,dashed] (0,0.25) -- (0.4,0.25) -- (0.4,1);
    \node at (0.7,0.7) {$\Pi^-$};
    \node at (0.325,0.325) {\scriptsize $\Lambda$};
    \node at (0.125,-0.1) {\tiny $x_1$};
    \node at (0.325,-0.1) {\tiny $x_2$};
    \node at (0.5,-0.3) {$\Pi^*$};
  \end{tikzpicture}
\end{center}
\caption{The permutons $\Pi', \Pi_y, \Pi_{x_2}, \Pi^+$ and $\Pi^*$ from the proof of Theorem~\ref{thm:counterexample}.
The symbol $\Pi^-$ in the figure represents the part of the permuton $\Pi$ not including the first and second layer,
and $\Lambda$ represents a scaled down copy of the permuton $\Lambda$.}
\label{fig:permutons}
\end{figure}
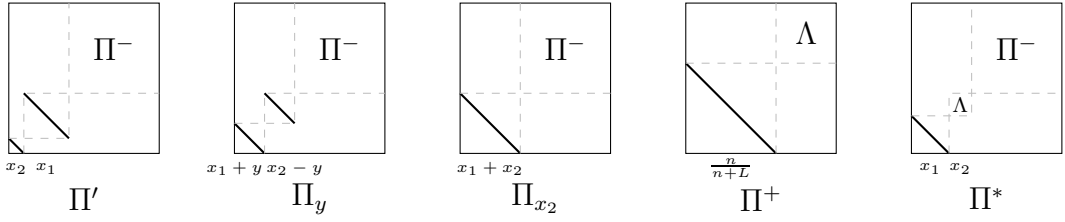

We start by showing that $x_1\geq n x_2$; the argument is split into two claims.

\begin{claim}
\label{claim:x_1>=x_2}
$x_1\geq x_2$.
\end{claim}

\begin{proofcl}[Proof of Claim~\ref{claim:x_1>=x_2}]
We consider the layered permuton $\Pi'$ obtained from $\Pi$ by swapping the first two layers,
i.e., the first layer of $\Pi'$ has length $x_2$, the second has length $x_1$, and
the $i$-th has length $x_i$ for $i=3,\ldots,K$.
Observe that
\[d(\sigma,\Pi')-d(\sigma,\Pi)=A\left(x_2^nx_1- x_1^nx_2\right)\Sigma_2.\]
Since $\Pi$ is $\sigma$-optimal, $A>0$ and $\Sigma_2>0$,
we conclude that $x_1\geq x_2$. 
\end{proofcl}

\begin{claim}
\label{claim:x_1/n}
$x_1\geq n x_2$.
\end{claim}

\begin{proofcl}[Proof of Claim~\ref{claim:x_1/n}]
For $y\in [0,x_2)$,
we consider the layered permuton $\Pi_y$
whose first layer has length $x_1+y$,
the second layer $x_2-y$,
and the $i$-th $x_i$ for $i = 3, \ldots, K$;
in particular, $\Pi_0$ is just the permuton $\Pi$.

Let $h(y)$ be the density of $\sigma$ in the permuton $\Pi_y$ (the density is viewed as a function of $y$), i.e.,
\[h(y)=d(\sigma,\Pi_y)=A\left((x_1+y)^n(x_2-y)\Sigma_2 + (x_1+y)^n\Sigma_1 + (x_2-y)^n\Sigma_1 + \Sigma_0\right).\]
Observe that
\begin{align*}
\frac{\partial h}{\partial y} & = A\left(n(x_1+y)^{n-1}(x_2-y)\Sigma_2 - (x_1+y)^n\Sigma_2 \right.\\
                              & + \left.n(x_1+y)^{n-1}\Sigma_1 - n(x_2-y)^{n-1}\Sigma_1\right),
\end{align*}
and that the value of the derivative of $h$ for $y=0$ is
\[A\left(nx_1^{n-1}x_2 - x_1^n\right)\Sigma_2 + A\left(nx_1^{n-1} - nx_2^{n-1}\right)\Sigma_1.\]
By Claim~\ref{claim:x_1>=x_2}, the second term in the above expression is non-negative. Therefore, 
the derivative of $h$ at $y=0$ is bounded below by
\[A\left(nx_1^{n-1}x_2 - x_1^n\right)\Sigma_2.\]
Since $\Pi$ is $\sigma$-optimal, the derivative of $h$ at $y=0$ must be non-positive. Therefore, $x_1\geq n x_2$. 
\end{proofcl}

We next bound $\Sigma_1$ from above as follows.

\begin{claim}
\label{claim:Sigma1}
$\Sigma_1 \leq \frac{x_1}{n}$.
\end{claim}

\begin{proofcl}[Proof of Claim~\ref{claim:Sigma1}]
We consider the layered permuton $\Pi_{x_2}$ obtained from $\Pi$ by merging its first and second layer;
formally, the permuton $\Pi_{x_2}$ is the layered permuton such that
its first layer has length $x_1+x_2$ and the $(i-1)$-th layer $x_i$ for $i=3,\ldots,K$.
The notation comes from Claim~\ref{claim:x_1/n} as $\Pi_{x_2}$ is the limit permuton of $\Pi_y$ for $y\to x_2$.
Observe that the difference
\begin{align*}
d(\sigma,\Pi_{x_2})-d(\sigma,\Pi) & = A\left(((x_1+x_2)^n - x_1^n-x_2^n)\Sigma_1 - x_1^nx_2\Sigma_2\right)\\
                            & \ge A\left(nx_1^{n-1}x_2\Sigma_1-x_1^nx_2\Sigma_2\right)
\end{align*}
must be non-positive since the permuton $\Pi$ is $\sigma$-optimal.
We use $A > 0$ and $\Sigma_2 \leq 1$ to obtain that $nx_1^{n-1}x_2\Sigma_1 \leq x_1^nx_2$,
which yields the bound in the inequality from the claim.
\end{proofcl}

Let $\sigma'$ be the layered permutation obtained from $\sigma$ by removing its first layer,
i.e., $\sigma'$ is the layered permutation with layers of sizes $(1,\ell_1,\dots,\ell_k)$.
The next two claims bound $d(\sigma,\Pi)$ and $x_1$ using the assumption that
there exists a layered permuton $\Lambda$ with a certain density of $\sigma'$.

\begin{claim}
\label{claim:DnotSmall}
For any layered permuton $\Lambda$, we have
\[d(\sigma,\Pi)\ge \binom{n+L}{n} \left(\frac{n}{n+L}\right)^n\left(\frac{L}{n+L}\right)^L d(\sigma',\Lambda).\]
\end{claim}

\begin{proofcl}[Proof of Claim~\ref{claim:DnotSmall}]
We consider the layered permuton $\Pi^+$ obtained from $\Lambda$ by adding a long first layer and rescaling.
Specifically, the first layer of the permuton $\Pi^+$ is $[0,n/(n+L)]$ and
if $[a,b]$ is a layer of $\Lambda$, then $[(n+aL)/(n+L),(n+bL)/(n+L)]$ is a layer of $\Pi^+$.
Since the density of $\sigma$ in $\Pi^+$ is at least
\[\binom{n+L}{n} \left(\frac{n}{n+L}\right)^n\left(\frac{L}{n+L}\right)^L d(\sigma',\Lambda)\]
and the permuton $\Pi$ is $\sigma$-optimal, the claim follows.
\end{proofcl}

Our final claim is the following.

\begin{claim}
\label{claim:x1Tiny}
For any layered permuton $\Lambda$ with $d(\sigma',\Lambda)>0$, we have
\[x_1\leq \left(\frac{A}{{n+L \choose n} n^{n-L} d(\sigma',\Lambda)}\right)^{1/L}.\]
\end{claim}

\begin{proofcl}[Proof of Claim~\ref{claim:x1Tiny}]
We consider the layered permuton $\Pi^*$ obtained from $\Pi$ by replacing the second layer with a scaled copy of~$\Lambda$,
i.e., the $i$-th layer of $\Pi$ for $i=1,3,\ldots,K$ is also a layer of $\Pi^*$ and
if $[a,b]$ is a layer of $\Lambda$, then $[x_1+ax_2,x_1+bx_2]$ is a layer of $\Pi^*$.
We use that the second layer of $\sigma$ has size one and derive that
\[d(\sigma,\Pi^*)-d(\sigma,\Pi)\ge {n+L \choose n} x_1^n x_2^L d(\sigma',\Lambda) - A x_2^n \Sigma_1\]
where the first term of right side accounts for the probability that
a $\Pi^*$-random permutation is $\sigma$,
its first layer is contained in the first layer of $\Pi^*$ and
its remaining layers in the layers of the scaled down copy of $\Lambda$, and
the second term accounts for the probability that
a $\Pi$-random permutation is $\sigma$ and
its first layer is contained in the second layer of $\Pi^*$.

Since $\Pi$ is $\sigma$-optimal, it holds that
\[{n+L \choose n} x_1^n x_2^L d(\sigma',\Lambda) \leq A x_2^n \Sigma_1,\]
and since $\Sigma_1 \leq 1$, we get
\[{n+L \choose n} x_1^n d(\sigma',\Lambda) \leq A x_2^{n-L}.\]
Using the assumption that $n \geq L$ and Claim~\ref{claim:x_1/n},
we derive that $x_1^{n-L} \geq (nx_2)^{n-L}$
and eventually obtain
\[{n+L \choose n} x_1^L n^{n-L} d(\sigma',\Lambda) \leq A,\]
which yields the desired inequality. 
\end{proofcl}

We next use \eqref{eq:Dexpanded} and $\Sigma_0 \leq \Sigma_1 \leq \Sigma_2 \leq 1$
to get that
\[d(\sigma,\Pi) \leq A\left(x_1^nx_2 + x_1^n+x_2^n + \Sigma_1\right).\]
Combining this with Claims~\ref{claim:x_1>=x_2} and~\ref{claim:Sigma1} and the fact that  $x_2 \leq 1$ yields
\begin{equation}
d(\sigma,\Pi) \leq A\left(3x_1^n + \frac{x_1}{n}\right).\label{eq:Dcontra}
\end{equation}
We next show that there exists $d_0>0$ such that $d(\sigma,\Pi)\ge d_0$ for all $n$.
Observe that
\[
\lim_{n\to\infty} \binom{n+L}{n} \left(\frac{n}{n+L}\right)^n\left(\frac{L}{n+L}\right)^L
 = \lim_{n\to\infty} \frac{L^L}{L!}\left(1-\frac{L}{n+L}\right)^n
 = \frac{L^L}{e^L L!}\,.\]
We fix any layered permuton $\Lambda$ such that $d(\sigma',\Lambda)>0$ and
obtain using Claim~\ref{claim:DnotSmall} that
the packing density of $\sigma$ for any sufficiently large $n$ is at least $\frac{L^L}{2e^L L!}d(\sigma',\Lambda)$;
the existence of $d_0$ now follows.

We next estimate the right side of \eqref{eq:Dcontra}.
Observe that $A=\frac{(n+L)!}{n!\prod_{j=1}^k\left(\ell_j!\right)}=\Theta(n^L)$
while $x_1=O(n^{-\frac{n}{L}})$ by Claim~\ref{claim:x1Tiny}.
We conclude that the right side of \eqref{eq:Dcontra} tends to $0$, and
choose $n_0$ in such a way that $n_0\ge L$ and the inequality \eqref{eq:Dcontra} does not hold for any $n\ge n_0$.
This finishes the proof of the theorem.
\end{proof}

The argument presented in Theorem~\ref{thm:counterexample}
can actually be used to obtain reasonably small bounds on $n_0$.
For example, consider the case when $k=1$ and $\ell_1=2$,
which means that $\sigma$ is the layered permutation with layers with sizes $(n,1,2)$ and
$\sigma'$ is the layered permutation with layers of sizes $(1,2)$, i.e., $\sigma'=132$.
Since the packing density of $\sigma'$ is $2\sqrt{3}-3$,
we can fix a permuton $\Lambda$ such that $d(\sigma',\Lambda)=2\sqrt{3}-3$.
As
\[A = \frac{(n+3)!}{n!\cdot 2} = 3{n+3 \choose n},\]
we get using Claims~\ref{claim:DnotSmall} and~\ref{claim:x1Tiny} that
\[d(\sigma,\Pi) \geq {n+3 \choose n}\left(\frac{n}{n+3}\right)^n\left(\frac{3}{n+3}\right)^3\left( 2\sqrt{3}-3\right)\]
and
\[x_1\leq \sqrt[3]{\frac{3}{n^{n-3}(2\sqrt{3}-3)}}\]
for any $\sigma$-optimal permuton $\Pi$ where $x_1$ is the length of the first layer of $\Pi$.
We show that it is possible to set $n_0=13$.
Observe that $(n+3) \leq 16n/13$ since $n\ge 13$.
It follows that
\[d(\sigma,\Pi)\geq{n+3 \choose n}\left(\frac{1}{e^3}\right)\left(\frac{3 \cdot 13}{16n}\right)^3\left(2\sqrt{3}-3\right)> 0.33{n+3 \choose n}\frac{1}{n^3}.\]
On the other hand, we obtain using $n\ge 13$ and $x_1\le 1$ that
\begin{align*}
A\left(3x_1^n + \frac{x_1}{n}\right) & \leq
3{n+3 \choose n}\left(\frac{3}{13^7(2\sqrt{3}-3)n^3} +
\frac{1}{n}\sqrt[3]{\frac{3}{13^4(2\sqrt{3}-3)n^6}}\right)\\
& < 0.19{n+3 \choose n}\frac{1}{n^3}.
\end{align*}
It follows that the inequality \eqref{eq:Dcontra} in the proof of Theorem~\ref{thm:counterexample} does not hold for any $n\ge 13$ and
so every $\sigma$-optimal layered permuton has infinitely many layers for every $n\ge 13$.

\section{Maximizers with finitely many layers}
\label{sec:finite}

This section is devoted to the proofs of Theorems~\ref{thm:finite} and~\ref{thm:constantBound2}.
We start with proving three lemmas needed to prove Theorem~\ref{thm:finite}.
Given a layered permuton $\Pi$,
we say that a non-trivial interval $[a,b]$ is a \emph{segment} of $\Pi$
if it can be written as a union of (possibly infinitely many) layers of $\Pi$. 

\begin{lemma}
\label{lm:segmentInLimit}
Let $\sigma$ be a layered permutation of order at least two with $k$ layers and without consecutive singleton layers.
If $\Pi$ is a $\sigma$-optimal layered permuton,
then every segment $[a,b]$ of $\Pi$ contains a layer of $\Pi$ of length at least $\frac{b-a}{|\sigma|^2k^{|\sigma|}}$. 
\end{lemma}

\begin{proof}
Suppose, for contradiction, that there exists a $\sigma$-optimal layered permuton $\Pi$ 
with a segment $[a,b]$ such that the length of every layer of $\Pi$ contained in $[a,b]$ 
is strictly less than $\frac{b-a}{|\sigma|^2k^{|\sigma|}}$.
Let $\Pi'$ be the layered
permuton obtained by replacing the layers of $\Pi$ in the segment $[a,b]$ with $k$ layers, each of 
length $(b-a)/k$. 
Using that $\sigma$ has no two consecutive singleton layers,
we argue that $d(\sigma,\Pi')>d(\sigma,\Pi)$.

Given $|\sigma|$ points sampled according to $\Pi$,
let $E_n$ be the event that exactly $n$ of the sampled points have their $x$-coordinate in $[a,b]$;
analogously,
$E'_n$ is the event that exactly $n$ of the sampled points according to $\Pi'$ have their $x$-coordinate in $[a,b]$.
Note that, since $\Pi$ and $\Pi'$ have uniform marginals,
it holds that $\PP(E_n)=\PP(E'_n)$ for all $n=0,\ldots,|\sigma|$.

Observe that the probability that
a $\Pi$-random permutation of order $|\sigma|$ is $\sigma$ conditioned on $E_0\cup E_1$
is equal to the probability that
a $\Pi'$-random permutation of order $|\sigma|$ is $\sigma$ conditioned on $E'_0\cup E'_1$.
We next analyze the probabilities that
a $\Pi$-random permutation of order $|\sigma|$ is $\sigma$ and
a $\Pi'$-random permutation of order $|\sigma|$ is $\sigma$
conditioned on a particular choice of $|\sigma|-n$ points outside the segment $[a,b]$ for $n\ge 2$.
Since the permutons $\Pi$ and $\Pi'$ agree outside the segment $[a,b]$,
the probabitity distributions of the events that we condition on are the same.
For a fixed choice of $|\sigma|-n$ points outside the segment $[a,b]$,
if the probability that
a $\Pi'$-random permutation of order $|\sigma|$ is $\sigma$ conditioned on the choice of the $|\sigma|-n$ points is zero,
then the probability that
a $\Pi$-random permutation of order $|\sigma|$ is $\sigma$ conditioned on the same choice of the points is also zero.
Otherwise,
the probability that
a $\Pi'$-random permutation of order $|\sigma|$ is $\sigma$ conditioned on the choice of the $|\sigma|-n$ points is at least $k^{-n}$ and
the probability that
a $\Pi$-random permutation of order $|\sigma|$ is $\sigma$ conditioned on the same choice of the points
is at most
\begin{equation}
\binom{|\sigma|}{2} \left(\frac{1}{|\sigma|^2k^{|\sigma|}}\right)<\frac{k^{-|\sigma|}}{2}.\label{eq:estim}
\end{equation}
The estimate \eqref{eq:estim} follows from the fact that $\sigma$ does not have two consecutive singleton layers and
so if the $\Pi$-random permutation is $\sigma$ and $n\ge 2$ points are sampled from the segment $[a,b]$,
then at least two points of $\sigma$ must come from the same layer contained in the segment $[a,b]$.
The probability of this event happening is upper bounded by the left side of \eqref{eq:estim}.
Hence, the probability that 
a $\Pi$-random permutation of order $|\sigma|$ is $\sigma$ conditioned on $E_2\cup\cdots\cup E_{|\sigma|-1}$
is at most the probability that
a $\Pi'$-random permutation of order $|\sigma|$ is $\sigma$ conditioned on $E'_2\cup\cdots\cup E'_{|\sigma|-1}$ (note that
the inequality is strict whenever the former probability is non-zero), and
the probability that
a $\Pi$-random permutation of order $|\sigma|$ is $\sigma$ conditioned on $E_{|\sigma|}$
is strictly smaller than the probability that
a $\Pi'$-random permutation of order $|\sigma|$ is $\sigma$ conditioned on $E'_{|\sigma|}$ (as these two probabilities are non-zero).
This contradicts the assumption that the permuton $\Pi$ is $\sigma$-optimal.
\end{proof}

A simple consequence of Lemma~\ref{lm:segmentInLimit} is the following.

\begin{lemma}
\label{lm:noTrivial}
Let $\sigma$ be a layered permutation of order at least two without consecutive singleton layers.
For every $\sigma$-optimal permuton $\Pi$,
the set of all $x\in [0,1]$ such that $[x,x]$ is a layer of $\Pi$ has measure zero.
\end{lemma}

\begin{proof}
Let $X$ be the set of all $x\in [0,1]$ such that $[x,x]$ is a layer of $\Pi$.
Observe that $X$ is closed, and so it is measurable.
Let $k$ be the number of layers of $\sigma$ and define $r=|\sigma|^2k^{|\sigma|}$.

By applying Lemma~\ref{lm:segmentInLimit} to the segment $[0,1]$,
we see that $\Pi$ has a layer $I_1$ of length at least $1/r$.
Let $I_0=[0,\min I_1]$ and $I_2=[\max I_1,1]$.
If $I_0$ is non-trivial, then it is a segment of $\Pi$ and we can apply Lemma~\ref{lm:segmentInLimit} to it;
similarly, if $I_2$ is non-trivial, we can apply Lemma~\ref{lm:segmentInLimit} to it.
In general,
for a finite sequence $i_1,i_2,\dots,i_t\in\{0,2\}$,
we recursively define $I_{i_1,\dots,i_t,0}$, $I_{i_1,\dots,i_t,1}$ and $I_{i_1,\dots,i_t,2}$
to be internally disjoint intervals whose union is $I_{i_1,\dots,i_t}$ such that
if $I_{i_1,\dots,i_t}$ is a segment, then $I_{i_1,\dots,i_t,1}$ is a layer of $\Pi$ of length at least $|I_{i_1,\dots,i_t}|/r$.
If $I_{i_1,\dots,i_t}$ is trivial, we set $I_{i_1,\dots,i_t,0}=I_{i_1,\dots,i_t,1}=I_{i_1,\dots,i_t,2}=I_{i_1,\dots,i_t}$.
For $t\geq1$, let
\[S_t = \bigcup_{i_1,\dots,i_t\in\{0,2\}} I_{i_1,\dots,i_t},\]
and observe that $X\subseteq \bigcap_{t=1}^\infty S_t$.
However, for each $t\geq 1$,
the measure of $S_{t+1}$ is at most $\frac{r-1}{r}$ times the measure of $S_t$.
We conclude that the set $X$ has measure zero. 
\end{proof}

Our next lemma deals specifically with the classes of permutations
considered in Theorem~\ref{thm:finite}. 

\begin{lemma}
\label{lm:shortPairInLimit}
Let $\varepsilon>0$ and let $\sigma$ be a layered permutation with $k\geq 2$ 
layers of sizes $(\ell_1,\dots,\ell_k)$ such that $\ell_1\ge 2$, $\ell_k\ge 2$ and 
$\ell_i+\ell_{i+1}\ge\max\{\ell_1,\ell_k\}+1$ for every $i\in [k-1]$. If $\Pi$ is a 
$\sigma$-optimal layered permuton with at least $2k-3$ layers of length at least $\varepsilon$, 
then every segment $[a,b]$ of $\Pi$ that has at least two layers
has a layer of length at least $(\varepsilon/4)^{|\sigma|}/k$.
\end{lemma}

\begin{proof}
Suppose, for the sake of contradiction, that $\Pi$ is a $\sigma$-optimal permuton that
has $2k-3$ layers of length at least $\varepsilon$ and
a segment $[a,b]$ with at least two layers such that
every layer contained in $[a,b]$ has length less than $(\varepsilon/4)^{|\sigma|}/k$. 

Let $\mathcal{I}$ be the set of all non-trivial layers of $\Pi$ contained in $[a,b]$.
Since the set $\mathcal{I}$ is countable,
we can index the elements of $\mathcal{I}$ by $I_j$ for $j\in S$,
where $S$ is either equal $\NN$ or $[n]$ for some $n\in\NN$.
For $j\in S$, let $z_j$ be the length of $I_j$.
Since $\ell_i+\ell_{i+1}\ge \max\{\ell_1,\ell_k\}+1\ge 3$,
the permutation $\sigma$ has no consecutive singleton layers and
so Lemma~\ref{lm:noTrivial} yields that $\sum_{j\in S}z_j=b-a$.

Consider a permuton $\Pi'$ obtained from $\Pi$ by replacing the segment $[a,b]$ with a single layer. 
We will show that $d(\sigma,\Pi')>d(\sigma,\Pi)$, contradicting that the permuton $\Pi$ is $\sigma$-optimal.
Since the permuton $\Pi$ has $2k-3$ layers each of length at least $\varepsilon$ and
none of these layers is contained in the segment $[a,b]$ (since each layer of the segment
is shorter than $(\varepsilon/4)^{|\sigma|}/k$),
there are $k-1$ layers of length at least $\varepsilon$ contained in $[0,a]$ or
there are $k-1$ layers of length at least $\varepsilon$ contained in $[b,1]$.
Since the two cases are completely symmetric, 
we analyze in detail the former case only.

Let $\sigma'$ be the permutation obtained from $\sigma$ by removing its last layer,
i.e., $\sigma'$ is the layered permutation with layers of sizes $(\ell_1,\ldots,\ell_{k-1})$.
Let $p$ be the probability that a $\Pi$-random permutation of order $|\sigma'|$ is $\sigma'$ and
the $x$-coordinates of all of its points belong to $[0,a]$.
Observe that $p\ge \varepsilon^{|\sigma'|}>\varepsilon^{|\sigma|}$. Note that $p$
is also equal to the probability that a $\Pi'$-random permutation of order $|\sigma'|$ is $\sigma'$ and
the $x$-coordinates of all of its points belong to $[0,a]$.
The probability that a $\Pi'$-random permutation of order $|\sigma|$ is $\sigma$ and
its last layer is sampled from $[a,b]$ is 
\[\binom{|\sigma|}{\ell_k} p (b-a)^{\ell_k} = \binom{|\sigma|}{\ell_k} p \left(\sum_{j\in S}z_j\right)^{\ell_k}.\]
Similarly, for $j\in S$, the probability that a $\Pi$-random permutation of order $|\sigma|$
is $\sigma$, its last layer is sampled from $I_j$, and all of its other layers are 
sampled from $[0,a]$ is equal to 
\[\binom{|\sigma|}{\ell_k} p z_j^{\ell_k}.\]
For $i\in [k-1]$, the probability that
a $\Pi$-random permutation is $\sigma$ and its $i$-th layer is sampled from the segment $[a,b]$ and
no other layer is sampled from the segment $[a,b]$
is at most the probability that
a $\Pi'$-random permutation is $\sigma$ and its $i$-th layer is sampled from the segment $[a,b]$ and
no other layer is sampled from the segment $[a,b]$.
Finally,
for $i\in [k-1]$, the probability that
a $\Pi$-random permutation is $\sigma$ and its $i$-th and $(i+1)$-th layers are
both sampled from the segment $[a,b]$ (note that these events are not disjoint for different $i$'s as
more than two layers can be sampled from $[a,b]$) is at most
\[\sum_{\substack{j_1,j_2\in S\\ j_1<j_2}}\binom{|\sigma|}{\ell_i}\binom{|\sigma|-\ell_i}{\ell_{i+1}}z_{j_1}^{\ell_i}z_{j_2}^{\ell_{i+1}}\le 2^{2|\sigma|}\sum_{\substack{j_1,j_2\in S\\ j_1<j_2}}z_{j_1}^{\ell_i}z_{j_2}^{\ell_{i+1}}.\]
Using the fact that $z_j<(\varepsilon/4)^{|\sigma|}/k$ for all $j\in S$, it follows that
\begin{align*}
d(\sigma,\Pi')-d(\sigma,\Pi) & \ge \binom{|\sigma|}{\ell_k}p\left(\left(\sum_{j\in S}z_j\right)^{\ell_k}-\sum_{j\in S}z_j^{\ell_k}\right)-2^{2|\sigma|}\sum_{i\in [k-1]}\sum_{\substack{j_1,j_2\in S\\ j_1<j_2}}z_{j_1}^{\ell_i}z_{j_2}^{\ell_{i+1}}\\
&                              \ge\varepsilon^{|\sigma|}\left(\left(\sum_{j\in S}z_j\right)^{\ell_k}-\sum_{j\in S}z_j^{\ell_k}\right)-4^{|\sigma|}\sum_{i\in [k-1]}\sum_{\substack{j_1,j_2\in S\\ j_1<j_2}}z_{j_1}^{\ell_i}z_{j_2}^{\ell_{i+1}}\\
&                              \ge\varepsilon^{|\sigma|}\sum_{\substack{j_1,j_2\in S\\ j_1<j_2}}\left(z_{j_1}^{\ell_k-1}z_{j_2} + z_{j_1}z_{j_2}^{\ell_k-1} \right)-4^{|\sigma|}\sum_{i\in [k-1]}\sum_{\substack{j_1,j_2\in S\\ j_1<j_2}}z_{j_1}^{\ell_i}z_{j_2}^{\ell_{i+1}}\\
&                        > 4^{|\sigma|}\left(k\sum_{\substack{j_1,j_2\in S\\ j_1<j_2}}\left(z_{j_1}^{\ell_k}z_{j_2} + z_{j_1}z_{j_2}^{\ell_k} \right)-\sum_{i\in [k-1]}\sum_{\substack{j_1,j_2\in S\\ j_1<j_2}}z_{j_1}^{\ell_i}z_{j_2}^{\ell_{i+1}}\right)\\
&			 \ge 4^{|\sigma|}\sum_{i\in [k-1]}\sum_{\substack{j_1,j_2\in S\\ j_1<j_2}}\left(z_{j_1}^{\ell_k}z_{j_2} + z_{j_1}z_{j_2}^{\ell_k}-z_{j_1}^{\ell_i}z_{j_2}^{\ell_{i+1}}\right).
\end{align*}
Since it holds $\ell_i+\ell_{i+1}\ge\ell_k+1$ for every $i\in [k-1]$ by the hypothesis,
we get that the following holds for all $i\in[k-1]$ and $j_1,j_2\in S$:
\begin{align*}
z_{j_1}^{\ell_k}z_{j_2}\geq z_{j_1}^{\ell_{i}}z_{j_2}^{\ell_{i+1}} &  \mbox{ if $z_{j_1} \geq z_{j_2}$, and}\\
z_{j_1}z_{j_2}^{\ell_k} \geq z_{j_1}^{\ell_{i}}z_{j_2}^{\ell_{i+1}} & \mbox{ if $z_{j_2} \geq z_{j_1}$.}
\end{align*}
This yields that
\[z_{j_1}^{\ell_k}z_{j_2} + z_{j_1}z_{j_2}^{\ell_k} > z_{j_1}^{\ell_i}z_{j_2}^{\ell_{i+1}} \geq 0.\]
It follows that $d(\sigma,\Pi')-d(\sigma,\Pi)>0$,
which contradicts the assumption that the permuton $\Pi$ is $\sigma$-optimal.
\end{proof}

We are now ready to prove Theorem~\ref{thm:finite}.

\begin{proof}[Proof of Theorem~\ref{thm:finite}]
Assume, for the sake of contradiction, that
there exists a $\sigma$-optimal layered permuton $\Pi$ with infinitely many layers, and
let $\varepsilon$ be the length of its $(2k-3)$-th longest layer;
note that $\varepsilon>0$ by Lemma~\ref{lm:noTrivial}.
By Lemma~\ref{lm:shortPairInLimit}, every segment $[a,b]$ of $\Pi$ is either
a layer of $\Pi$ or it contains a layer of length at least $(\varepsilon/4)^{|\sigma|}/k$. 
This implies that the number of layers of $\Pi$ is finite,
in particular, the permuton $\Pi$ has no trivial layers, and
that the number of layers of $\Pi$ is at most $2 \floor{k(4/\varepsilon)^{|\sigma|}} + 1$;
this contradicts the assumption that $\Pi$ has infinitely many layers.
\end{proof}

We next turn our attention to proving Theorem~\ref{thm:constantBound2},
which will follow from the next two lemmas.

\begin{lemma}
\label{lm:klayers}
Let $\sigma$ be a layered permutation with $k$ layers and without
consecutive singleton layers. Then every $\sigma$-optimal permuton has at least
$k$ layers of length at least $\frac{1}{|\sigma|^3k^{2|\sigma|+1}}$.
\end{lemma}

\begin{proof}
Let $\Pi$ be a $\sigma$-optimal layered permuton. Clearly, $\Pi$ 
has at least $k$ layers as otherwise $d(\sigma,\Pi)=0$. 
Since the packing density of $\sigma$ is at least $k^{-|\sigma|}$,
it holds that $d(\sigma,\Pi)\ge k^{-|\sigma|}$.

Let $r$ be the sum of the lengths of the $k-1$ largest layers of $\Pi$.
Observe that $d(\sigma,\Pi)\le |\sigma|(1-r)$ (because every $\Pi$-random permutation that is $\sigma$ contains a point outside of the $k-1$ largest layers of $\Pi$),
which implies that $\frac{1}{|\sigma|k^{|\sigma|}}\le 1-r$.
Let $[a,b]$ be the longest segment of $\Pi$ which does not contain any
of the $k-1$ longest layers, and note that its length is at least $(1-r)/k$. 
By Lemma~\ref{lm:segmentInLimit}, this segment contains a layer of length at least
\[\left(\frac{1-r}{k}\right)\left(\frac{1}{|\sigma|^2k^{|\sigma|}}\right)\geq \frac{1}{|\sigma|^3k^{2|\sigma|+1}}.\]
The statement of the lemma now follows.
\end{proof}

The second lemma needed to prove Theorem~\ref{thm:constantBound2} is the following.

\begin{lemma}
\label{lm:mergepair}
Let $\sigma$ be a layered permutation with layers of sizes $(\ell_1,\dots,\ell_k)$ and $k\ge 2$ such that
$\ell_1=\ell_k\ge 2$ and every two consecutive layers of $\sigma$ contain a layer of size $\ell_1=\ell_k$.
For every $C\in (0,1)$, there exists $c\in (0,C)$ such that
the following holds for every layered permutation $\pi$ of order at least $|\sigma|C^{-1}$ that
has at least $k$ layers of size at least $C|\pi|$.
If $\pi$ contains two consecutive layers of size at most $c|\pi|$ and
$\pi'$ is a permutation obtained from $\pi$ by merging any two such layers,
then $d(\sigma,\pi')\ge d(\sigma,\pi)$.
In addition,
if the merged layer has size at least $\ell_1$, then $d(\sigma,\pi')>d(\sigma,\pi)$.
\end{lemma}

\begin{proof}
Fix a permutation $\sigma$ with the properties given in the statement of the lemma and $C\in (0,1)$.
We prove that the statement of the lemma holds for $c\in (0,C)$,
which is determined at the very end of the proof and which depends on $\sigma$ and $C$ only.
Let $\pi$ be a permutation with the properties given in the statement of the lemma,
let $\pi_a$ and $\pi_{a+1}$ be the layers of $\pi$ that are merged to obtain $\pi'$, and
let $\alpha, \beta\in (0,1)$ be such that $|\pi_a|=\alpha|\pi|$ and $|\pi_{a+1}|=\beta|\pi|$.
In addition, for $i\in[1,k]$,
let $p_i$ be the number of occurrences of the layered permutation with layers $(\ell_1,\ldots,\ell_i)$ in $\pi$ before $\pi_a$, and
let $q_i$ be the number of occurrences of the layered permutation with layers $(\ell_i,\ldots,\ell_k)$ after $\pi_{a+1}$; and let $p_0 = q_{k+1} = 1$. 

We next count the occurrences of $\sigma$ in $\pi$ and $\pi'$ similarly to the proofs of Claim~\ref{claim:Sigma1} and Lemma~\ref{lm:shortPairInLimit}.  
Considering which layers of $\sigma$ are sampled from $\pi_a$ and $\pi_{a+1}$,
we rewrite ${|\pi|\choose|\sigma|} (d(\sigma,\pi')-d(\sigma,\pi))$ as
\begin{eqnarray}
&\sum_{i=1}^k p_{i-1}q_{i+1}\left({|\pi_a|+|\pi_{a+1}|\choose \ell_i}-{|\pi_a|\choose \ell_i}-{|\pi_{a+1}|\choose \ell_i}\right) -\nonumber\\
&\sum_{i=1}^{k-1}p_{i-1}q_{i+2}{|\pi_a|\choose \ell_i}{|\pi_{a+1}|\choose \ell_{i+1}}.\label{eq:mergep}
\end{eqnarray}

First note that if $|\pi_a|+|\pi_{a+1}|\le\ell_1$, then the second sum
is zero as $\ell_i=\ell_1$ or $\ell_{i+1}=\ell_1$ for every $i\in[1,k-1]$.
So, if $|\pi_a|+|\pi_{a+1}|\le\ell_1$,
then the conclusion of the lemma follows since it holds that
\[{|\pi_a|+|\pi_{a+1}|\choose \ell_i}\geq  {|\pi_a|\choose \ell_i}+{|\pi_{a+1}|\choose \ell_i}\]
for every $i\in [k]$.
Hence, we will assume that $|\pi_a|+|\pi_{a+1}|\ge \ell_1+1$ in the rest of the proof.

We next give a lower bound on the first sum in \eqref{eq:mergep}.
We say that a layer of $\pi$ is \emph{big} if its size is at least $C|\pi|$;
note that $C|\pi|\ge |\sigma|$.
Let $j$ be such that there are at least $j-1$ big layers before the layer $\pi_a$,
at least $k-j$ big layers after the layer $\pi_{a+1}$, and $\ell_j=\ell_1$.
It follows that
\[p_{j-1}q_{j+1}\geq  \prod_{i\neq j} {\lceil C|\pi|\rceil\choose \ell_i}\geq  \prod_{i\neq j} \left(\frac{C|\pi|}{\ell_i}\right)^{\ell_i}=(C|\pi|)^{|\sigma|-\ell_1}\prod_{i\neq j}\ell_i^{-\ell_i}.\]
Since the first sum in \eqref{eq:mergep} is at least its $j$-th term,
we obtain that the first sum in \eqref{eq:mergep} is at least
\begin{align*}
  &(C|\pi|)^{|\sigma|-\ell_1}\prod_{i\neq j}\ell_i^{-\ell_i}\left({{|\pi_a|+|\pi_{a+1}|\choose \ell_1}}-{{|\pi_a|\choose \ell_1}-{|\pi_{a+1}|\choose \ell_1}}\right)\\
  &\geq  (C|\pi|)^{|\sigma|-\ell_1}\prod_{i\neq j}\ell_i^{-\ell_i} \times \frac{2|\pi_a||\pi_{a+1}|}{\ell_1(\ell_1-1)}{|\pi_a|+|\pi_{a+1}|-2\choose \ell_1-2}\\
  &> (C|\pi|)^{|\sigma|-\ell_1}\prod_{i\neq j}\ell_i^{-\ell_i} \times \frac{2|\pi_a||\pi_{a+1}|}{\ell_1^2}\left(\frac{|\pi_a|+|\pi_{a+1}|-2}{\ell_1}\right)^{\ell_1-2}\\
  &\geq  (C|\pi|)^{|\sigma|-\ell_1}\prod_{i\neq j}\ell_i^{-\ell_i} \times \frac{2|\pi_a||\pi_{a+1}|}{\ell_1^2}\left(\frac{|\pi_a|+|\pi_{a+1}|}{3\ell_1}\right)^{\ell_1-2}\\
  &\geq  A |\pi|^{|\sigma|}\alpha\beta\left(\alpha+\beta\right)^{\ell_1-2}\ge A |\pi|^{|\sigma|}\left(\alpha^{\ell_1-1}\beta + \alpha\beta^{\ell_1-1}\right)
\end{align*}
where $A=\frac{2C^{|\sigma|-\ell_1}}{(3\ell_1)^{\ell_1}\prod_{i\in [k]}\ell_i^{\ell_i}}$ (note that $A$ depends on $\sigma$ and $C$ only). 

We next find an upper bound on the second sum in \eqref{eq:mergep}.
Let $i\in[1,k-1]$, and
observe that $p_{i-1}q_{i+2}$ is at most the number of ways of choosing $|\sigma|-\ell_i-\ell_{i+1}$
points from a permutation of order $|\pi|-|\pi_a|-|\pi_{a+1}|$,
i.e., $p_{i-1}q_{i+2}\leq {|\pi|-|\pi_a|-|\pi_{a+1}|\choose|\sigma|-\ell_i-\ell_{i+1}}$.
Therefore, the $i$-th summand of the second sum in \eqref{eq:mergep} is at most 
\begin{align*}
  &{|\pi|-|\pi_a|-|\pi_{a+1}|\choose|\sigma|-\ell_i-\ell_{i+1}}{|\pi_a|\choose \ell_i}{|\pi_{a+1}|\choose \ell_{i+1}} \\
  &\leq  \left(\frac{e(|\pi|-|\pi_a|-|\pi_{a+1}|)}{|\sigma|-\ell_i-\ell_{i+1}}\right)^{|\sigma|-\ell_i-\ell_{i+1}}  \left(\frac{e|\pi_a|}{\ell_i}\right)^{\ell_i}\left(\frac{e|\pi_{a+1}|}{\ell_{i+1}}\right)^{\ell_{i+1}}\\
  &\leq  B |\pi|^{|\sigma|} (1-\alpha-\beta)^{|\sigma|-\ell_i-\ell_{i+1}}\alpha^{\ell_i}\beta^{\ell_{i+1}}\\
  &\leq  B |\pi|^{|\sigma|} \alpha^{\ell_i}\beta^{\ell_{i+1}} \leq  B |\pi|^{|\sigma|} (\alpha^{\ell_1}\beta + \alpha\beta^{\ell_1})
\end{align*}
where $B=\max_{i\in [k]}\frac{e^{|\sigma|}}{\ell_i^{\ell_i}\ell_{i+1}^{\ell_{i+1}}(|\sigma|-\ell_i-\ell_{i+1})^{|\sigma|-\ell_i-\ell_{i+1}}}$ (again $B$ depends on $\sigma$ only).

We combine the bounds on the first and the second sum in \eqref{eq:mergep} to obtain that
\begin{align*}
d(\sigma,\pi')-d(\sigma,\pi)
& > {|\pi|\choose|\sigma|}^{-1}|\pi|^{|\sigma|}\left(A (\alpha^{\ell_1-1}\beta + \alpha\beta^{\ell_1-1}) - kB (\alpha^{\ell_1}\beta + \alpha\beta^{\ell_1})\right) \\
& \geq \left(\frac{|\sigma|!}{|\pi|^{|\sigma|}}\right)|\pi|^{|\sigma|}\left(A  (\alpha^{\ell_1-1}\beta + \alpha\beta^{\ell_1-1}) - kB (\alpha^{\ell_1}\beta + \alpha\beta^{\ell_1})\right) \\
& \geq A  (\alpha^{\ell_1-1}\beta + \alpha\beta^{\ell_1-1}) - kB (\alpha^{\ell_1}\beta + \alpha\beta^{\ell_1}).
\end{align*}
Hence, the lemma holds with $c=\min\{A/kB,C\}/2$.
\end{proof}

We are now ready to prove Theorem~\ref{thm:constantBound2}.

\begin{proof}[Proof of Theorem~\ref{thm:constantBound2}]
Fix a layered permutation $\sigma$ of order $m$ and with layers of sizes $(\ell_1,\dots,\ell_k)$, $k\ge 2$, such that
$\ell_1=\ell_k\ge 2$ and every two consecutive layers of $\sigma$ contain a layer of size $\ell_1=\ell_k$.
In particular, $\sigma$ does not have consecutive singleton layers.
Apply Lemma~\ref{lm:mergepair} with $C=\frac{1}{2m^3k^{2m+1}}$ to get a constant $c\in (0,C)$.

For the sake of contradiction, suppose that for every $n\in\NN$
there exists a $\sigma$-optimal permutation $\pi_n$ with at least $n$ layers.
Without loss of generality,
we can assume also assume that $|\pi_n|\ge |\sigma|/C$ for every $n$ and that
the sequence $(\pi_n)_{n\in\NN}$ is convergent (as any sequence of permutations with increasing orders has a convergent subsequence), and
let $\Pi$ be the limit permuton of $(\pi_n)_{n\in\NN}$.
Note that the permuton $\Pi$ is layered (by Lemma~\ref{lm:convergence}) and 
the density of $\sigma$ in $\Pi$ is the packing density of $\sigma$.
Lemma~\ref{lm:klayers} implies that the permuton $\Pi$ has $k$ layers of length at least $2C$.
Hence, there exists $n_0\in\NN$ such that every permutation $\pi_n$, $n\ge n_0$,
has $k$ layers of size at least $C|\pi_n|$ (cf.~the proof of Lemma~\ref{lm:convergence}).

For $n\ge n_0$, define $\pi'_n$ to be a layered permutation obtained from the permutation $\pi_n$
by successively merging pairs of consecutive layers of size at most $c|\pi_n|$ as long as
two consecutive layers of size at most $c|\pi_n|$ exist.
Observe that the number of layers of $\pi'_n$ is at most $\floor{2/c}+1$.
By Lemma~\ref{lm:mergepair}, it holds that $d(\sigma,\pi'_n)\ge d(\sigma,\pi_n)$, and
since $\pi_n$ is a $\sigma$-optimal permutation,
it actually holds that $d(\sigma,\pi'_n)=d(\sigma,\pi_n)$.
So, Lemma~\ref{lm:mergepair} yields that no layer of $\pi'_n$ obtained by merging two or more layers of $\pi_n$
has size at least $\ell_1$.
It follows that the number of layers of $\pi_n$ is at most $\left(\floor{2/c}+1\right)(\ell_1-1)$,
which is impossible if $n>\left(\floor{2/c}+1\right)(\ell_1-1)$.
The statement of the theorem follows.
\end{proof}

\section{Conclusion}
\label{sec:concl}

We conclude with five open problems
focusing on the nature of the layered maximizers for a layered permutation 
$\sigma$.
The first question asks whether
the existence of near-$\sigma$-optimal sequence of layered permutations with a bounded number of layers
implies that all $\sigma$-optimal layered permutations have a bounded number of layers.

\begin{problem}
\label{prob2}
Is there a layered permutation $\sigma$ such that
there exists a near-$\sigma$-optimal sequence of layered permutations with a bounded number of layers
but for every $k\in\NN$ there exists a $\sigma$-optimal layered permutation (of some order) with at least $k$ layers?
\end{problem}

In principle, the existence of a $\sigma$-optimal layered permutation with a bounded number of layers
might depend on the order of a host permutation for some layered permutations $\sigma$.

\begin{problem}
\label{probinf}
Is there a layered permutation $\sigma$ and an integer $m$ such that
for infinitely many $n\in\NN$,
there exists a $\sigma$-optimal layered permutation of order $n$ with at most $m$ layers, and
still for every $k\in\NN$, there exists $n\in\NN$ such that
every $\sigma$-optimal layered permutation of order $n$ has at least $k$ layers?
\end{problem}

A $\sigma$-optimal permutation for a layered permutation $\sigma$ need not be layered in general.
For example, Albert et al.~\cite[Remark~2.3]{AlbAHHS02} observed that, for $\sigma=43215$, the permutation
\[16,15,\ldots,2,1,17,19,20,18\]
is $\sigma$-optimal, but not layered.
However, we do not know whether there exists such an example in the limit setting.

\begin{problem}
\label{prob0}
Is there a layered permutation $\sigma$ such that
there exists a $\sigma$-optimal permuton that is not layered?
\end{problem}

The final two problems are limit analogues of Problem~\ref{prob2}.

\begin{problem}
\label{prob3}
Is there a layered permutation $\sigma$ such that
there exist both a $\sigma$-optimal permuton with finitely many layers and a $\sigma$-optimal permuton with infinitely many layers?
\end{problem}

\begin{problem}
\label{prob4}
Is there a layered permutation $\sigma$ such that
every $\sigma$-optimal permuton has finitely many layers
but there exist $\sigma$-optimal permutons with an arbitrarily large number of layers?
\end{problem}

\section*{Acknowledgement}
We thank the anonymous referee for their helpful comments.

\bibliographystyle{bibstyle}
\bibliography{layer3}

\end{document}